\documentclass[reqno]{amsart}
\usepackage{amsmath,amssymb,bbm, color, amsthm}
\usepackage[backref,colorlinks]{hyperref}
\usepackage[all,cmtip]{xy}
\usepackage{verbatim}

\newtheorem{theorem}{Theorem}

\newtheorem{lemma}[theorem]{Lemma}

\newtheorem{claim}[theorem]{Claim}

\newtheoremstyle{definition}
  {4pt}
  {4pt}
  {\sl}
  {}
  {\bfseries}
  {.}
  {.5em}
  {}

\theoremstyle{definition}
\newtheorem{definition}[theorem]{Definition}

\theoremstyle{remark}
\newtheorem{remark}[theorem]{Remark}

\newtheoremstyle{introthms}
  {3pt}
  {3pt}
  {\itshape}
  {}
  {\bfseries}
  {.}
  {.5em}
  {\thmnote{#3}}
\theoremstyle{introthms}

\newcommand{\eps}{\varepsilon}

\newcommand{\cS}{\mathcal{S}}

\newcommand{\ind}{\mathrm{ind}}

\renewcommand{\l}{\ensuremath{\ell}}

\begin{document}
\title{A regularity lemma and twins in words} 
\author[M. Axenovich]{Maria Axenovich}
\address{Iowa State University, Ames, U.S.A. and Karlsruher Institut f\"ur Technologie, Karlsruhe, Germany
}
\email{maria.aksenovich@kit.edu}
\thanks{The research of the first author is supported in part by NSF  grant DMS-0901008.}

\author[Y. Person]{Yury Person}
\address{Freie Universit\"at Berlin, Institut f\"ur Mathematik, 
 Berlin, Germany} 
\email{person@math.fu-berlin.de}
\author[S. Puzynina]{Svetlana Puzynina}
\address{University of Turku,  Turku,  Finland, and Sobolev Institute of Mathematics, Novosibirsk, Russia}
\email{svepuz@utu.fi}
\thanks{The research of the third author is supported in part by grant 251371 of the Academy of Finland.} 

\date{\today}
\begin{abstract}
For a word $S$, let $f(S)$ be the largest integer $m$ such that
there are  two disjoints identical (scattered) subwords of length
$m$. Let $f(n, \Sigma) = \min \{f(S):   S  \mbox{  is of length }
n, \mbox{ over alphabet } \Sigma \}$. Here, it is shown  that
$$2f(n, \{0,1\}) = n-o(n)$$ using the regularity lemma for words.
I.e.,  any binary word of length $n$ can be split into
two identical subwords (referred to as twins) and, perhaps, a
remaining subword of length $o(n)$. A similar result is proven for $k$
identical subwords of a word over an alphabet with at most $k$
letters.
\end{abstract}

\maketitle


\section{Introduction}

Let $S=s_1 \ldots s_n $ be a word of length $n$, i.e., a sequence $s_1, s_2, \ldots, s_n$.   
A (scattered) {\it subword} of
$S$ is a word $S'= s_{i_1} s_{i_2} \ldots s_{i_s}$, where
$i_1<i_2<\cdots < i_s$. This notion was largely investigated in
combinatorics on words and formal languages theory with  special
attention  given to counting subword occurrences, different
complexity questions, the problem of reconstructing a word from
its subwords (see, e.g., \cite{ds2003, MS2004, salo2003}). For a
word $S$, let $f(S)$ be the largest integer $m$ such that there
are two disjoints identical subwords of $S$, each of  length $m$.
We call such subwords {\it twins}. For example, 
if $S=s_1 s_2
s_3 s_4 s_5 s_6 = 0 0 1 0 1 1$, then $S'=s_1 s_5$ and $S_2=s_4
s_6$ are two identical subwords  equal to $0 1$. The question we
are concerned with is "How large could the twins  be
in any word over a given alphabet?" One of the classical problems
related to this question   is the problem of finding longest
subsequence common to two given sequences, see for example
\cite{CLRS01, H, Xia07}. Indeed,  if we split a given word $S$
into two subwords with the same number of elements and find a
common to these two subwords  word, it would correspond to
disjoint identical subwords  in $S$. Optimizing over all
partitions gives largest twins.

Denoting   $\Sigma^{n}$  the set of  words of length $n$ over the alphabet $\Sigma$, let
 \[f(n, \Sigma) = \min \{f(S):   S \in \Sigma^n \}.\]
Observe  first, that $f(n, \{0,1\})\geq\lfloor(1/3)n\rfloor$.
Indeed, consider any $S\in \Sigma^n$ and
 split  it  into consecutive triples. Each triple has either two zeros or two ones, so we can build a subword $S_1$
by choosing a repeated element from each triple, and similarly
build a subword $S_2$ by choosing the second repeated element from
each triple. For example, if $S = 001~101~111~010$ then there are
twins  $S_1, S_2$,  each  equal to $0~1~1~0$:        $S = {\bf
0}{\color{red} 0}1~{\bf 1} 0{\color{red} 1}~{\bf 1}{\color{red}
1}1~{\bf 0}1{\color{red} 0}$,  here one word is marked bold, and
the other marked red.

In fact, we can find much larger identical subwords in any binary word.
 Our main result is
\begin{theorem}\label{thm:twins}
There exists an absolute constant $C$ such that
\[
\left(1-C\left(\frac{\log n}{\log\log n}\right)^{-1/4}\right)n \leq     2f(n, \{0,1\})\leq    n - \log n.
\]
\end{theorem}
In the proof we shall employ a classical density increment
argument successfully applied in combinatorics and number theory,
see e.g.\ the survey of Koml\'os and Simonovits~\cite{KoSi96} and some important applications~\cite{Gow01} and~\cite{Szemeredi78}. We first show that
we can partition any word $S$ into consecutive factors  that look as
if they were random in a certain weak sense (we call them
$\eps$-regular). These $\eps$-regular words can be partitioned  (with the exception of $\eps$ proportion of letters) into two
identical subwords.  By appending these together for
every $\eps$-regular word, we eventually obtain identical subwords  of
roughly  half the length of $S$.

We generalize the notion of two identical subwords in words to a notion of
$k$ identical subwords. For a given word $S$,  let $f(S, k)$ be the largest
$m$ so that $S$ contains   $k$ pairwise 
disjoint identical subwords of length $m$ each. Finally,
let
$$f(n, k, \Sigma) = \min \{f(S, k):  ~ S\in \Sigma^n    \}.$$

\begin{theorem}\label{thm:k- tuplets-smallQ}
For any integer $k\geq 2$, and alphabet $\Sigma$, $|\Sigma|\leq
k$,
\[
\left(1-C|\Sigma|\left(\frac{\log n}{\log\log n}\right)^{-1/4}\right)n  \leq     kf(n, k, \Sigma).
\]
\end{theorem}

 In case when $k$ is smaller than the size of the alphabet, we have the following bounds.
 \begin{theorem}\label{thm:k- tuplets-largeQ}
For any integer $k\geq 2$, and alphabet $\Sigma$, $|\Sigma|> k$,
\[
\left(\frac{k}{|\Sigma|}-C|\Sigma|\left(\frac{\log n}{\log\log n}\right)^{-1/4}\right)n \leq     kf(n, k,
\Sigma)\leq n - \max\{\alpha n,\log n\}, 
\]
where $\alpha\in[0,1/k]$ is the solution of the equation $\l^{-(k-1)\alpha} \alpha ^{-k \alpha }(1-k\alpha )^{k\alpha-1 }= 1$,
 whenever such solution exists and $0$ otherwise.
\end{theorem}

We shall sometimes refer to two disjoint identical subwords as {\it twins},
three disjoint identical subwords as {\it triplets},  $k$ disjoint identical subwords as $k$-{\it tuplets}.
We shall prove the regularity lemma for binary words in Section
\ref{regularity}  and will prove the Theorem~\ref{thm:twins} in
Section \ref{Proofs}.     We shall prove Theorems
\ref{thm:k- tuplets-smallQ}, \ref{thm:k- tuplets-largeQ}  in
Section~\ref{Proofs1}. We shall ignore   any
divisibility issues as these will not affect our arguments.

\section{Definitions and Regularity Lemma for Words}\label{regularity}

 First, we shall introduce some notations (for more detail, see for instance \cite {KK, Lo}).
 An \emph{alphabet} $\Sigma$ is a finite non-empty set of symbols called \emph{letters}.
  For a (scattered) {\it subword}  $S'= s_{i_1} s_{i_2}
\ldots s_{i_s}$,  of  a word $S$,     we call the set
$\{i_1, i_2, \ldots, i_s\}$ a {\it support} of $S'$ in $S$, and
write ${\rm supp}(S')$, so the length of $S'$,  $|S'|= |{\rm supp}(S')|$. 
Denoting  $I= \{i_1, \ldots, i_s\}$, we write $S' = S[I]$.   
A {\it factor} of $S$ is a subword with consecutive elements of
$S$, i.e., $s_i s_{i+1} \ldots s_{i+m}$, for some $1\leq i \leq n$
and $0\leq m \leq n-i$, we denote it $S[i, i+m]$. If $S$ is a word
over alphabet $\Sigma$ and $q\in \Sigma$, we denote $|S|_q$ the
number of elements of $S$ equal to $q$. The {\it density}  $d_q(S)
$ is defined to be $|S|_q/|S|$.

For two subwords $S'$ and $S''$ of $S$, we say that $S'$ is
contained in $S''$ if ${\rm supp}(S')\subseteq {\rm supp}(S'')$,
we also denote by $S'\cap S''$ a subword of $S$,  $S[{\rm
supp}(S')\cap {\rm supp}(S'')]$.
 If $S= s_1 \ldots s_n$ and
$S[1,i] = A$, $S[i+1, n]=B$, then we write $S=AB$ and call $S$ a
concatenation of $A$ and $B$.

 \begin{definition}[$\eps$-regular word]
Call a word $S$ of length $n$ over an alphabet $\Sigma$
$\eps$-regular if for every $i$,  $\eps n+1 \le  i\le n-2\eps n+1$ and
every $q\in \Sigma$ it holds that
\begin{equation}\label{eq:irregular}
 |d_q(S)-d_q(S[i,i+\eps n-1])|<\eps.
\end{equation}
\end{definition}

Notice that in the case $|\Sigma|=|\{0,1\}|=2$,  $d_0(S)=1-d_1(S)$ and thus
$ |d_0(S)-d_0(S[i,i+\eps n-1])|<\eps \Longleftrightarrow |d_1(S)-d_1(S[i,i+\eps n-1])|<\eps.$
When $\Sigma=\{0,1\}$, we shall denote $d(S) = d_1(S)$.

The notion of $\eps$-regular words resembles the notion of pseudorandom (quasirandom) word, see ~\cite{ChungGrahamZ92}.
However, these two notions are quite different.  A word that  consists of alternating
$0$s and $1$s is   $\eps$-regular  but not pseudorandom.  Also, unlike in the case
of stronger notions of  pseudorandomness, one can check  in  a linear time whether a word is $\eps$-regular,
cf.~\cite{ADLRY94} in the graph case.

\begin{definition}
We call  $\cS:=(S_1$, \ldots, $S_t)$ a partition of $S$ if $S= S_1S_2\ldots S_t$, ($S$ is concatenation of consecutive $S_i$s). 
A partition $\cS$  is an  $\eps$-regular partition of a word $S\in\Sigma^n$    if 
\[
\sum_{\substack{i\in[t]\\ S_i \text{ is not } \eps-\text{regular}}} |S_i|\le \eps n,
\]
\noindent i.e., the total length  of $\eps$-irregular subwords is
at most $\eps n$.
\end{definition}

The decomposition lemma we are going to show states the following: 
\begin{theorem}[Regularity Lemma for Words]\label{lem:rl_seqs}
 For every $\eps>0$ and $t_0$ there is an $n_0$ and $T_0$ such that any word $S\in \Sigma^n$, for $n\geq n_0$
admits an $\eps$-regular partition of $S$ into $S_1$, \ldots,
$S_t$ with $t_0\le t\le T_0$. In fact, $T_0\le t_03^{1/{\eps^4}}$
and $n_0 =  t_0\eps^{-\eps^{-4}}$.
\end{theorem}

To prove the regularity lemma, we introduce the notion of an index and a refinement and prove a few basic facts.
\begin{definition}[Index of a partition]
Let $\cS:=(S_1$, \ldots, $S_t)$ be a partition of $S\in \Sigma^n$
into consecutive factors.
We define
\[
 \ind(\cS)= \sum_{q\in \Sigma }  \sum_{i\in[t]}  d_q(S_i)^2\tfrac{|S_i|}{n}.
\]
Further, for convenience we set $\ind_q(\cS)=\sum_{i\in[t]}  d_q(S_i)^2\tfrac{|S_i|}{n}$.
\end{definition}
Observe that $\ind(\cS)$ is bounded by $1$ from above.

\begin{definition}[Refinement of $\cS$]
 Let $\cS=(S_1,\ldots,S_t)$ and
 \[
\cS'=(S'_{1,1}, S'_{1,2}, \ldots, S'_{1, s_1},   \quad  S'_{2,1}, S'_{2,2}, \ldots, S'_{2, s_2}, \quad  \ldots, \quad S'_{t,1}, S'_{t,2}, \ldots, S'_{t, s_t})
\]
  be partitions of $S\in \Sigma^n$.
We say that $\cS'$ refines $\cS$ and write $\cS'\preccurlyeq\cS$,
if for every $i=1, \ldots, t$, $S_i = S'_{i,1} S'_{i,2} \cdots
S'_{i, s_i}$.
%
\end{definition}

\begin{lemma}\label{fact:indexineq}
 Let $\cS$ and  $\cS'$  be partitions of $S\in \Sigma^n$ 
If $\cS'\preccurlyeq\cS$ then
\[
 \ind(\cS')\ge\ind(\cS).
\]
\end{lemma}
\begin{proof}
Let $\cS=(S_1,\ldots,S_t)$  and
\[
\cS'=(S'_{1,1}, S'_{1,2}, \ldots, S'_{1, s_1},   \quad  S'_{2,1}, S'_{2,2}, \ldots, S'_{2, s_2}, \quad  \ldots, \quad S'_{t,1}, S'_{t,2}, \ldots, S'_{t, s_t}).
\]
We proceed for each $q\in \Sigma$ as follows:
 \begin{align*}
   \ind_q(\cS')& =  \sum_{S'\in\cS'} d_q(S')^2\frac{|S'|}{n}\\
                   &=  \sum_{i=1}^t  \sum_{j=1}^{s_i}   d_q(S'_{i,j})^2\frac{|S'_{i,j}|}{n}\\
                  &=   \sum_{i=1}^{t}  \frac{|S_i|}{n}  \sum_{j=1}^{s_i}  d_q(S_{i,j})^2\frac{|S'_{i,j}|}{|S_i|} \\
                  & \overset{\text{Jensen's inequality}}{\ge}  \sum_{i=1}^t\frac{|S_i|}{n} \left(\sum_ {j=1}^{s_i} d_q(S'_{i,j})\frac{|S'_{i,j}|}{|S_i|}\right)^2 \\
                 & = \sum_{i=1}^t\frac{|S_i|}{n} \left(\sum_ {j=1}^{s_i} \frac{|S'_{i,j}|_q}{|S'_{i,j}|} \frac{|S_{i,j}|}{|S_i|}\right)^2\\
                  &=\sum_{i=1}^t\frac{|S_i|}{n} d_q(S_i)^2 \\
                  &=\ind_q(\cS).
\end{align*}
Now, building the sum over all $q\in \Sigma$ yields:
\[
 \ind(\cS')\ge \ind(\cS).
\]
\end{proof}
The next lemma shows that if a word $S$ is not
$\eps$-regular, then there is a  refinement of $(S)$  whose index exceeds the index of $(S)$ by at least  $\eps^3$.
\begin{lemma}\label{lem:increment}
 Let $S\in \Sigma^m$ be an $\eps$-irregular word. Then there is a  partition  $(A,B, C)$ of $S$ such that $|A|, |B|, |C| \geq \eps m$ and
\begin{equation}\label{eq:increment}
 \ind((A,B,C))\ge\ind((S))+\eps^3=\left(\sum_{q\in \Sigma} d_q(S)^2\right)+\eps^3.
\end{equation}
\end{lemma}
\begin{proof}
 Since $S$ is not $\eps$-regular, there exists an element $q\in \Sigma$ and an $i$ with $\eps m+1\le i\le m-2\eps m+1$  such that
$|d-d(S[i,i+\eps m-1])|\ge \eps$, where $d:=d_q(S)$ and
$d(T):=d_q(T)$ for any factor $T$ of $S$. Assume w.l.o.g.\ that
$d-d(S[i,i+\eps m-1])\ge \eps$ and set $\gamma:=d-d(S[i,i+\eps
m-1])$, $A:=S[1,i-1]$, $B:=S[i,i+\eps m-1]$ and $C:=S[i+\eps
m,m]$, $a:=|A|$, $b:=|B|=\eps m$ and $c:=|C|$.

Observe further that
\[
|S|_q=d(A)a+d(B)b+d(C)c=dm,  \quad d((A,C))=\tfrac{dm-(d-\gamma)b}{a+c}, \quad    d(B)=d-\gamma.
\]
Since  $a+c = m-b$ and $\ind_q((A,B,C))=\ind_q((A,C,B))$, 
\begin{align*}
 \ind_q((A,B,C))&{\ge}  d((A,C))^2\frac{a+c}{m}+d(B)^2\frac{b}{m} \\
&= \left(\frac{dm-(d-\gamma)b}{a+c}\right)^2\frac{a+c}{m}+(d-\gamma)^2\frac{b}{m}  \\
& =   \frac{(dm-(d-\gamma)b)^2}{(m-b)m}+(d-\gamma)^2\frac{b}{m}  \\
&= \frac{1}{(m-b)m}  \left[  d^2 (m^2 -mb) + \gamma^2(mb)\right] \\
&= d^2 + \frac{\gamma^2 b}{m-b}
\geq d^2 + \frac{\eps^3 m}{(1-\eps)m}
\geq  d^2 + \eps^3.
\end{align*}
The case when $d-d(S[i,i+\eps n-1])\le -\eps$ works out similarly. Indeed, set $\gamma:=d-d(S[i,i+\eps m-1])$
as before and notice that $|\gamma|\ge\eps$ and all the computations above are exactly the same.

So, $\ind_q ((A,B,C)) \geq d_q^2  + \eps^3$. 
For all other  $q' \in \Sigma$,  Lemma~\ref{fact:indexineq} gives that $\ind_{q'} ((A,B,C)) \geq \ind_{q'}((S)) = d_{q'}^2(S)$.
Thus  $$\ind((A,B,C))=  \ind_q((A,B,C)) +\sum_{q'\in \Sigma -\{q\} }  \ind_{q'} ((A,B,C))  \geq  \sum_{q'\in \Sigma} d_{q'}(S)^2 + \eps^3.$$
\end{proof}

Finally we are in position to finish the argument.
\begin{proof}[Proof of the Regularity Lemma for Words] 
 Take $\eps>0$ and $t_0$ as given. We will give a bound on $n_0$ later.
Suppose that we have a word $S\in \Sigma^n$. Split it into $t_0$
consecutive factors $S_1$, \ldots, $S_{t_0}$ of the same length
$\tfrac{n}{t_0}$. If $\cS:=(S_1,\ldots ,S_{t_0})$ 
is not an
$\eps$-regular partition, then let  $I\subseteq [t_0]$ be the set
of all indices such that, for every $i\in I$, $S_i$ is not
$\eps$-regular (thus, $\sum_{i\in I}|S_i|\ge \eps n$). Then, by
Lemma~\ref{lem:increment}, we can refine each $S_i$, $i\in I$,
into factors $A_i$, $B_i$ and $C_i$ such that
$\ind((A_i,B_i,C_i))\ge \sum_{q\in \Sigma}d_q(S_i)^2+\eps^3$ (in
the case that~\eqref{eq:irregular} is violated for several $q\in
\Sigma $, choose an arbitrary  such $q$). We perform such
refinement for each $S_i$, $i\in I$, obtaining a partition
$\cS'\preccurlyeq\cS$, noticing that
\begin{align*}
 \ind(\cS')&=
  \sum_{q\in \Sigma}\sum_{j\in[t_0]\setminus I} d_q(S_j)^2\frac{|S_j|}{n}+
 \\ &\quad\quad\quad\quad\quad
   \sum_{q\in \Sigma }\sum_{i\in I} \left(d_q(A_i)^2\frac{|A_i|}{n}+d_q(B_i)^2\frac{|B_i|}{n}+d_q(C_i)^2\frac{|C_i|}{n}\right)  \\
  &= \sum_{q\in \Sigma}\sum_{j\in[t_0]\setminus I} d_q(S_j)^2\frac{|S_j|}{n}+ \sum_{i\in I} \ind((A_i,B_i,C_i))\frac{|S_i|}{n} \\
  &\overset{\eqref{eq:increment}}{\ge} \sum_{q\in \Sigma }\sum_{j\in[t_0]\setminus I} d_q(S_j)^2\frac{|S_j|}{n}+\sum_{i\in I}(\ind((S))+\eps^3)\frac{|S_i|}{n} \\
  &=\ind(\cS)+\eps^3\frac{\sum_{i\in I}|S_i|}{n} \\
  &\ge \ind(\cS)+\eps^4.
\end{align*}
Thus, $\cS'$ refines $\cS$ and has higher index. If $\cS'$ is not an $\eps$-regular partition of $S$, then we can repeat the procedure above by refining $\cS'$ etc.
Recall that an index of any partition $\cS$ is bounded from above by $1$.
Thus, since the increment  of the index that we get at each step is at least $\eps^4$ and each word in the partition decreases in length by a factor of  at most $\eps$ at each step, it
follows that we can perform at most $\eps^{-4}$ many steps so that the resulting factors are non-trivial,  and therefore we will eventually find an $\eps$-regular partition of $S$.
Notice that such a partition consists of at most  $3^{1/\eps^4}t_0$ words, since at each iteration each of the words is partitioned into at most $3$ new ones.
Therefore, $T_0\le 3^{1/\eps^4}t_0$ and each factor  in the partition  has  length  at least  $t_0^{-1}\eps^{1/\eps^4}n$.
\end{proof}

\section{Proof of  Theorem~\texorpdfstring{\ref{thm:twins}}{twins}.}\label{Proofs}
Before we prove our main theorem about binary words, we show a useful claim about twins in $\eps$-regular words.
\begin{claim}\label{claim:psrandom}
 If $S$ is an $\eps$-regular word, then $2f(S)\geq |S| - 5\eps|S|$.
\end{claim}
\begin{proof}
Let $|S|=m$. We partition $S$ into $t= 1/\eps$ consecutive factors
$S_{1}$,\ldots, $S_{1/\eps}$, each of length $\eps m$. Since $S$
is $\eps$-regular, $|d(S_{i})-d(S)|<\eps$,   for every
$i\in\{2,\ldots, 1/\eps-1\}$. Thus each $S_i$ has at least $(d(S)-
\eps)\eps m$ occurrences of $1$s and at least $(1-d(S)-\eps)\eps m$ occurrences of $0$s.
Let $S_i(1)$ be a subword of  $S_i$ consisting of exactly $(d(S)-
\eps)\eps m$ letters $1$ and $S_i(0)$ be a subword of $S_i$ consisting
of exactly  $(1-d(S)-\eps)\eps m$ letters $0$.   Consider the following two
disjoint subwords of $S$:
$A=S_2(1)S_3(0) S_4(1) \cdots S_{t-2}(1) $ and   $B=S_3(1) S_4(0) S_5(1) \cdots S_{t-2}(0)S_{t-1}(1) $.
 When $t$ is odd, $A$ and $B$ are constructed similarly.

We see that $A$ and $B$ together have at least $m - 2\eps^2 m ( 1/\eps -3) -3 \eps m$ elements,
where $2\eps^2 m ( 1/\eps -3)$ is an upper bound on the number of $0$s and $1$s which we had to ``throw away''
to obtain \emph{exactly} $(d(S)- \eps)\eps m$ letters  $1$ and $(1-d(S)-\eps)\eps m$ letters $0$  in each $S_i$, 
 $2\eps m$   is the number of elements in   $S_1$ and $S_t$, and $\eps m$ is the upper bound on  $|S_2(0)|+|S_{t-1}(1)|$.
Thus,    $2f(S) \geq m - 5\eps m $.  This concludes the proof of the claim.
\end{proof}

Notice that we could slightly improve on $5\eps m$ above by finding in an already mentioned way twins of size $\eps m/3$ each in $S_1$ and $S_t$,
but this does not give great improvement.

\begin{proof}[Proof of  Theorem~\ref{thm:twins}]
Let $n$ be at least $n_0$, which is as asserted by the Regularity
Lemma for words
for  given $\eps>0$ and $t_0:=\lceil\tfrac{1}{\eps}\rceil$.
Furthermore, let $S$ be a binary word of length $n$. Again,
Theorem ~\ref{lem:rl_seqs} asserts an $\eps$-regular partition of
$S$ into $S_1$, \ldots, $S_t$ with $1/\eps \le t\le T_0$. We apply
Claim~\ref{claim:psrandom}  to every $\eps$-regular factor $S_i$.
Furthermore, since $S_i$s appear consecutively in
$S$, we can put the twins from each of $S_i$s together obtaining
twins for the whole word $S$. This way we see:
\[
 2f(S)\ge \sum_{\substack{i\in[t]\\ S_i \text{ is }\eps-\text{regular}}} (|S_i|-5\eps |S_i|)
\ge n-5\eps n-\eps n=n-6\eps n,
\]
here $\eps n$ corresponds to the total lengths of not $\eps$-regular factors.
Choosing $\eps=C(\frac{\log n}{\log\log n})^{-1/4}$,  and an appropriate $C$, we see that
$n \geq  \eps^{-\eps^{-4}}$.
 Therefore, by Theorem \ref{lem:rl_seqs} ~~$2f(n,\{0,1\})\ge (1-C(\log n)^{-1/4}))n.$ 
%
%
%
%

 Next we shall prove  the upper bound on $f(n,\{0,1\})$ by constructing  a binary word $S$ such that $2f(S) \leq |S| - \log |S|$.
 Let  $S= S_k S_{k-1} \ldots S_0$, where $|S_i|= 3^i$,  $S_i$ consists only of $1$s for even $i$, and
it consists only of $0$s for odd $i$s. I.e., $S$ is built of iterated $1$- or $0$-blocks
exponentially decreasing in size.  Let $A$ and $B$ be twins in $S$. \\
\noindent
Assume first that  $A$ and $B$ have the same number of elements in $S_k$.
Since $S_k$ has odd number of elements, and $A$, $B$ restricted to $S'=S_{k-1}S_{k-2} \cdots S_0$ are twins,
by induction we have that $|A|+|B| \leq (|S_k|-1) + (|S'| - \log (|S'|))  = |S| -  1 -\log(|S'|)  \leq |S| - \log |S|.$
That is true   since $|S_k| = 3^k$, $|S| = (3^{k+1}-1)/2$. \\
\noindent
Now assume, w.l.o.g.\,  that  $A$ has more elements in $S_k$ than $B$ in $S_k$.
Then $B$ has no element in $S_{k-1}$.
We have  that  $| A \cap S_{k-1} | \geq |S_{k-1}|/2$,
otherwise $|A|+|B| \leq |S| - |S_{k-1}|/2 \leq |S| - \log |S|$.
 So,  $s=|A\cap S_{k-1}| \geq |S_{k-1}|/2 \geq  3^{k-1}/2$, and
$s$ elements  of $B$ must be in $S_{k-3}\cup S_{k-5} \cdots$.
But $|S_{k-3}|+ |S_{k-5}| + \cdots   \leq 3^{k-2}/2$, a contradiction proving Theorem \ref{thm:twins}.
\end{proof}

\begin{remark}
One can find words of length $n/2-o(n)$ as described above by an
algorithm with $O(\eps^{-4}|Q|n)$ steps.
\end{remark}

\section{\texorpdfstring{$k$}{k}-tuplets over alphabet of at most \texorpdfstring{$k$}{k} letters}\label{Proofs1}

\begin{proof}[Proof of Theorem \ref{thm:k- tuplets-smallQ}]
As before, we concentrate first on $\eps$-regular words. Let  $S$
be an $\eps$-regular word of length $m$  over alphabet
$\Sigma=\{0,\ldots,{\l}-1\}$ and recall the assumption $\l\le k$. We
partition $S$ in $t= 1/\eps$ consecutive factors $S_{1}$,\ldots,
$S_{1/\eps}$, each of length $\eps m$. Since $S$ is
$\eps$-regular, $|d_q(S_{i})-d_q(S)|<\eps$,   for every
$i\in\{2,\ldots, 1/\eps-1\}$, and every $q\in \Sigma $. Thus $S_i$ has
at least $(d_q(S)- \eps)\eps m$ letters  $q$, for each $q\in \Sigma$.

We construct $k$-tuplets $A_1$, \ldots, $A_k$ as
follows. Each of $A_j$s consists of consecutive blocks, with
first block consisting  of $(d_{0}(S)- \eps)\eps m$ letters $0$,
followed by a block of  $(d_{1}(S)- \eps)\eps m$ letters $1$, \ldots,
followed by a block of  $(d_{{\l}-1}(S)- \eps)\eps m$ letters ${\l}-1$, followed
by a block of $(d_{0}(S)- \eps)\eps m$ letters $0$, and so on.

%
Since $k\ge |\Sigma|$,
we will use  all but at most
$\tfrac{1}{\eps} \eps^2m |\Sigma|+(2|\Sigma|)\eps m = 3|\Sigma| \eps m$ elements, 
where the first
summand accounts for the number of elements that we did not use
when choosing exactly $(d_{q}(S)- \eps)\eps m$ elements $q$
from each $S_i$ and each $q\in \Sigma$  and the second summand for the number of elements
in   $S_1$, \ldots, $S_{\l}$,
and from $S_{1/\eps-\l+1}$, \ldots, $S_{1/\eps}$.

Below are the examples in the special cases when $|\Sigma|=\l=k$ and when $|\Sigma|=2$ and $k=4$.

\noindent
{\it Example 1.}
\begin{align*}
A_1&=S_2(0) S_3(1) S_4(2) \cdots S_{\ell+1}(\ell-1) S_{\ell+2}(0) S_{\ell+3}(1) \cdots S_{2\ell+1} (\ell-1) \cdots, \\
 A_2&= \,\,\, S_3(0) S_4(1) S_5(2) \cdots S_{\ell+2}(\ell-1) S_{\ell+3}(0) S_{\ell+4}(1) \cdots S_{2\ell+2} (\ell-1) \cdots,  \\
\vdots\\
 A_i&= \quad S_{i+1}(0) S_{i+2}(1) S_{i+3}(2) \cdots S_{i+\ell}(\ell-1) S_{i+\ell+1}(0) S_{i+\ell+2}(1) \cdots S_{i+2\ell} (\ell-1) \cdots\\
\vdots\\
 A_{k}&= \hskip 2.4 cm S_{\l+1}(0) S_{\l+2}(1) S_{\l+3}(2) \cdots S_{2\ell}(\ell-1) S_{2\ell+1}(0)  \cdots S_{3\ell} (\ell-1) \cdots
\end{align*}

\noindent
{\it Example 2.}
\begin{align*}
A_1&=S_2(0) S_3(1) \hskip 1.6cm  S_6(0) S_7(1)   \cdots  \\
 A_2&= \hskip 0.8 cm S_3(0) S_4(1)  \hskip 1.6cm     S_7(0) S_8(1)\cdots \\
 A_3&= \hskip 1.6 cm S_{4}(0) S_{5}(1)  \hskip 1.6cm   S_8(0) S_9(1) \cdots  \\
 A_4&= \hskip 2.4 cm S_{5}(0) S_{6}(1)   \hskip 1.6cm  S_9(0) S_{10}(1) \cdots
\end{align*}

Here $S_i(j)$  is the block of $(d_j(S)-\eps)\eps m$ letters $j$ taken from $S_i$.   So, in general, the total number of elements in $A_1$,\ldots, $A_k$ is at least
$m-3|\Sigma|\eps m.$
Thus, $k f(S)\ge m-3 |\Sigma|\eps m$.

To provide the lower bound on $f(n,k, \Sigma )$ we proceed as in the
proof of Theorem \ref{thm:twins}  by first finding a regular partition of a given
word and then applying the above construction to regular factors with an appropriate choice of $\eps$.
\end{proof}

\section{Large alphabets and small \texorpdfstring{$k$}{k}-tuplets}
\begin{proof}[Proof of Theorem~\ref{thm:k- tuplets-largeQ}]
The proof of the lower bound proceeds by considering a scattered word $W$ consisting of the $k$ most frequent letters.
Clearly, $|W|\ge \tfrac{k}{|\Sigma|}n$, which together with Theorem~\ref{thm:k- tuplets-smallQ} yields the lower bound.

The upper bound we obtain is either immediate from Theorem~\ref{thm:twins} or from computing the expected number of $k$-tuplets of length $m$ each
in a random word of length $n$ over an alphabet $\Sigma$ of size $\l$. If the expectation if less than $1$, this
means that there is a word $S$ with $f(S,k)<m$. Indeed, there are
\[
 \frac{1}{k!} \prod_{i=0}^{k-1} \binom{n-im}{m}
\]
distinct  sets  of $k$  disjoint subwords each of length $m$  in a word of length $n$.
The probability that  such a set corresponds  to a   $k$-tuplet,  when each letter is chosen with probability $1/\l$ independently,
is $\l^{(1-k)m}$.
Thus, the expected number of  $k$- tuplets  is at most
\begin{multline*}
 \l^{(1-k)m}\prod_{i=0}^{k-1} \binom{n-im}{m}=\l^{-(k-1)m} \frac{n!}{(m!)^k(n-km)!}\le \l^{-(k-1)m}\frac{n^n}{m^{km}(n-km)^{n-km}},
\end{multline*}
that is, for $m=\alpha n$, is at most
\[
 \l^{-(k-1)\alpha n}\frac{n^n}{(\alpha n)^{k\alpha n }(n-k\alpha n)^{n-k\alpha n}}=\left(\l^{-(k-1)\alpha} \alpha ^{-k\alpha }(1-k\alpha )^{k\alpha-1 }\right)^n.
\]
Thus, if $\l^{-(k-1)\alpha}    \alpha ^{-k\alpha }   (1-k\alpha )^{k\alpha-1 }$ is less than $1$ then $f(S,k)\le \alpha n$.
In particular, for $k=2$ and $\l=5$ one can compute that $\alpha<0.49$.
\end{proof}

\section{Concluding Remarks}
\subsection{Small values of \texorpdfstring{$f(n,k,\Sigma)$}{f(n,k,Sigma)}}
\begin{table}[h]
\begin{tabular}{|c|c|c|c|c|c|c|c|c|c|c|c|c|}
\hline
$\Sigma$\textbackslash $n$ &$6$&$7$&$8$&$9$&$10$&$11$&$12$&$13$&$14$&$15$&$16$&$17$ \\\hline
$\{0,1\}$        &$2$&$2$&$2$&$3$&$3$ &$ 4$&$ 4$&$ 5$&$ 5$&$ 5$& $6$   & $6$       \\\hline
$\{0,1,2\}$      & $1$  &$1$   &$2$   &$2$   &$2$    &$3$    &$ 3$&$ 3$&$ 4$&$ 4$&$ 4$&$ 4$ \\
\hline
\end{tabular}
\end{table}
\begin{table}[h]
\begin{tabular}{|c|c|c|c|c|c|c|c| }
\hline
$\Sigma$\textbackslash $n$ & $18$&$19$&$20$&$21$&$22$&$23$&$24$\\\hline
$\{0,1\}$            &$7$    &$7$    &$ 8$&    &    &    &    \\\hline
$\{0,1,2\}$      &$ \leq 5$ &$\le6$&$\le6$&$\le 7$&$\le 7$&${\color{red} \le 8}$&$\le8$\\
\hline
\end{tabular}
\caption{Values for small $t$ of $f(t,2,2)$ and $f(t,2,3)$.}
\end{table}

We will slightly abuse notation and denote by $f(n,k,\l)$ the value of $f(n,k,\Sigma)$ with $|\Sigma|=\l$.
In the introductory section it was observed that $f(3,2,2)=1$ yielding immediately a weak lower bound on $f(n,2,2)$ to be $\lfloor n/3\rfloor$.
In general, it holds clearly
\[
 f(n,k,\l)\ge \left\lfloor\tfrac{n}{m}\right\rfloor f(m,k,\l).
\]
For example, we determined (Theorem~\ref{thm:k- tuplets-largeQ}) a lower bound on $f(n,2,3)$ to be $\tfrac{1}{3}n-o(n)$.
We do not know whether it is tight and, more sadly, whether one can achieve it, without $o(n)$ term, by finding a (reasonable) number $t$
such that $f(t,2,3)\ge \frac{t}{3}$. If one could find such $t$ this would immediately give another proof of $f(n,2,3)\ge\tfrac{1}{3}n-t$.
However, the smallest value for such possible $t$ could be $21$, which already presents a computationally challenging task. In the tables above we summarize estimates on the values on $f(n,k,\l)$,
which were determined with the help of a computer. Thus, the first ``open'' case which might improve
lower bound on $f(n,2,3)$ is $f(22,2,3)$.
\subsection{Improving the \texorpdfstring{$O\left(|\Sigma|\left(\frac{\log\log n}{\log n}\right)^{1/4}\right)n$}{o(n)} term}
Further we remark, that a more careful analysis below of the increment argument in the proof of Theorem~\ref{lem:rl_seqs} leads to the bound $T_0\le t_0 3^{(-2\log \eps)/\eps ^3}$, which in turn improves the
bounds in Theorems~\ref{thm:twins} and~\ref{thm:k- tuplets-smallQ} to
\[
\left(1-C|\Sigma|\left(\frac{(\log\log n)^2}{\log n}\right)^{1/3}\right)n  \leq     kf(n, k, \Sigma).
\]

Recall that in the proof of Theorem~\ref{lem:rl_seqs} we set up an index and refining a corresponding partition 
each time we increase it by at least $\eps^4$. Let's reconsider $j$th refinement step at which
the partition $\cS=(S_1,\ldots,S_{t_0})$ is to be refined. Further recall that $I$ consists of the indices  $i$ such that
$S_i$ is not $\eps$-regular. 
 Let $\alpha_j$ be such that 
\begin{equation}\label{eq:alpha}
\sum_{i\in I} |S_i|=\alpha_j n.
\end{equation}
In the original proof we iterate as long as $\alpha_j\ge \eps$ holds. And by peforming an iteration step we 
merely use the fact that $\alpha_j\ge \eps$ which leads to $\eps^4$ increase of the index during one iteration step. 
Recall that $\ind(\cS)$ was defined as follows:
\[
\ind(\cS)=  \sum_{q\in \Sigma }\sum_{j\in[|\cS|]} d_q(S_j)^2\frac{|S_j|}{n},
\]
and for each further refinement $\cS'\preccurlyeq\cS$ it holds:
\begin{equation}\label{eq:newbound_ind}
\ind(\cS)\le \ind(\cS')=\frac{(1-\alpha_j)n}{n}\ind(\cS_1)+\frac{\alpha_j n}{n}\ind(\cS_2) \le \sum_{q\in \Sigma }\sum_{j\in[|\cS|]\setminus I} d_q(S_j)^2\frac{|S_j|}{n}+\alpha_j,
\end{equation}
where $\cS_1$ consists of $\eps$-regular words from $\cS$(these words are not partitioned/refined anymore)
and $\cS_2$ consists of not $\eps$-regular words from $\cS$ (and their lengths sum up to $\alpha_j n$).

Let $\l$ be the total number of iteration steps until we arrive at an $\eps$-regular partition. Let 
$\alpha_1$, \ldots, $\alpha_{\l}$ be the numbers, where $\alpha_j n$ is the sum over the lengths of  not $\eps$-regular 
words in the partition at step $j$, $j\in[\l]$ (cf.\eqref{eq:alpha}). 

By the discussion above
\[
1\ge \alpha_1\ge \alpha_2\ge \ldots\ge \alpha_{\l}\ge \eps.
\]
Next, we partition $(\eps,1]$ into $\log_2 \tfrac{1}{\eps}$ consecutive intervals $(y_{i+1},y_i]$ where $y_1=1$ and 
$y_{i+1}=y_i/2$. We claim that 
each interval $(y_{i+1},y_i]$ contains at most $\frac{2}{\eps^3}$ $\alpha_j$s. Indeed, the increase of the index during step $j$ where $\alpha_j\in(y_{i+1},y_i]$ is at least 
\[
\alpha_j \eps^3>y_{i+1}\eps^3.
\]
Further, let $j'$ be the smallest index such that $\alpha_{j'}\le y_i$ and $j''$ be the largest 
index such that  $\alpha_{j''}>y_{i+1}$. Let $\ind_j$ be the index before the $j$th refinement step.
Then by~\eqref{eq:newbound_ind} the following holds for $j'+1\le j\le j''$:
\[
\ind_{j'+1}\le \ind_{j}\le \ind_{j''}\le \ind_{j'+1}+y_i.
\]
This implies that the number of $\alpha_j$s in the interval $(y_{i+1},y_i]$ cannot be bigger than 
\[
\frac{y_i}{y_{i+1}\eps^3}=\frac{2}{\eps^3}.
\]
Thus, we obtain the following upper bound on $\l$ 
\[
\l\le \frac{2\log_2 \tfrac{1}{\eps}}{\eps^3},
\]
which leads to $T_0\le t_0 3^{(-2\log \eps)/\eps ^3}$, $n_0=t_0\eps^{-(2\log 1/\eps)/\eps^3}$ and thus we can regularize with 
$\eps=\left(\frac{(\log\log n)^2}{\log n}\right)^{1/3}$. 

\section*{Acknowledgements}  The authors would like to thank Sergey Avgustinovich  for fruitful discussions.


\end{document}